\documentclass[11pt,letterpaper,reqno]{amsart}
\usepackage{tikz}
\usetikzlibrary{positioning, shapes.geometric, arrows.meta, calc, positioning}
\usepackage{amssymb}
\usepackage{amsmath}
\usepackage{amsthm}
\usepackage{amsfonts}
\usepackage{bbm}
\usepackage{enumitem}
\usepackage{pgfplots}
\pgfplotsset{compat=1.18}
\usepackage{booktabs}
\usepackage{graphicx}
\usepackage[T1]{fontenc}
\usepackage{doi}
\usepackage{float}
\usepackage{bookmark}
\usepackage{hyperref}
\hypersetup{pdfstartview={FitH}}

\newtheorem{thm}{Theorem}[section]
\newtheorem{lem}[thm]{Lemma}

\newtheorem{claim}{Claim}

\newtheorem{conj}[thm]{Conjecture}
\theoremstyle{definition}

\numberwithin{equation}{section}

\begin{document}

\title{Proof of a conjectured spectral upper bound on the chromatic number of a graph}

\author[Q.~Tang]{Quanyu Tang}
\author[C.~Elphick]{Clive Elphick}

\address{School of Mathematics and Statistics, Xi'an Jiaotong University, Xi'an 710049, P. R. China}
\email{tang\_quanyu@163.com, tangquanyu827@gmail.com}
\address{School of Mathematics, University of Birmingham, Birmingham, UK}
\email{clive.elphick@gmail.com}



\begin{abstract}
Let $G$ be a simple graph on $n$ vertices and $m$ edges with chromatic number $\chi$, and let $\lambda_n$ denote the least adjacency eigenvalue. Solving a conjecture of Fan, Yu and Wang~[Electron. J. Combin., 2012], we prove that when $3\le \chi\le n-1$, the chromatic number satisfies the following upper bound:
$$
\chi \le \left(\frac{n}{2}+1+\lambda_n\right) +
\sqrt{\left(\frac{n}{2}+1+\lambda_n\right)^{2}-4(\lambda_n+1)\left(\lambda_n+\frac{n}{2}\right)},
$$
with equality if and only if $G \cong \left(K_{\frac{\chi}{2}}\cup\tfrac{n-\chi}{2}K_1\right) \vee \left(K_{\frac{\chi}{2}}\cup\tfrac{n-\chi}{2}K_1\right)$, where both $n$ and $\chi$ are even. This extends the validity of Fan--Yu--Wang's bound from the range $3\le \chi\le \frac{n}{2}$ to the full range $3\le \chi\le n-1$.

We also compare this bound with the well-known bound due to Wilf that $\chi \le 1 + \lambda_1$, where $\lambda_1$ denotes the largest eigenvalue. In particular we show that while Wilf's bound is an upper bound for some parameters larger than $\chi$, this bound using $\lambda_n$ is not an upper bound for these parameters. We conclude with a similar conjectured upper bound for $\chi(G)$, which uses $m$ in place of $n$.  
\end{abstract}

\maketitle

\section{Introduction}

Let $G=(V,E)$ be a simple graph of order $n$ with vertex set $V=V(G)$ and edge set $E=E(G)$. The \emph{adjacency matrix} of $G$ is the $n\times n$ $0$--$1$ matrix $A(G)=[a_{ij}]$, where $a_{ij}=1$ if $v_i$ is adjacent to $v_j$ and $a_{ij}=0$ otherwise. 
The eigenvalues of $A(G)$ are called the \emph{eigenvalues} of the graph $G$. 
Since $A(G)$ is symmetric, all eigenvalues are real; we list them as
\[
\lambda_1(G)\ge \lambda_2(G)\ge \cdots \ge \lambda_n(G).
\]

The \emph{chromatic number} of a graph $G$, denoted by $\chi(G)$, is the least number of colors needed to color $V(G)$ so that adjacent vertices receive distinct colors. There has been extensive research on the connection between eigenvalues and the chromatic number of a graph. Wilf~\cite{W67} proved that $\chi(G)\le 1+\lambda_1(G)$. For connected graphs, equality holds if and only if $G$ is a complete graph or an odd cycle. Hoffman~\cite{H70} showed that $\chi(G)\ge 1-\frac{\lambda_1(G)}{\lambda_n(G)}$. Elphick, Tang and Zhang~\cite{ETZ26} proved lower bounds on  chromatic numbers in terms of $p$-energies, which substantially generalize the Hoffman bound.

We use the following notation.
Let $K_n$ (resp.~$O_n$) denote the complete (resp.~empty) graph on $n$ vertices; note that $O_n = nK_1$, the disjoint union of $n$ isolated vertices. For two vertex-disjoint graphs $H_1=(V_1,E_1)$ and $H_2=(V_2,E_2)$, their \emph{disjoint union} is 
\[
H_1\cup H_2\quad\text{with}\quad V(H_1\cup H_2)=V_1\cup V_2,\ \ E(H_1\cup H_2)=E_1\cup E_2,
\]
and their \emph{join} is
\[
H_1\vee H_2\quad\text{with}\quad V(H_1\vee H_2)=V_1\cup V_2,\ \ 
E(H_1\vee H_2)=E_1\cup E_2\cup\{uv:\ u\in V_1,\ v\in V_2\}.
\]
(Equivalently, every vertex of $H_1$ is made adjacent to every vertex of $H_2$.) 
We also write $K_{a,b}$ for the complete bipartite graph with parts of sizes $a$ and $b$; note that $K_{a,b}=O_a\vee O_b$.

It is immediate that $\chi(G)=n$ if and only if $G=K_n$; $\chi(G)=1$ if and only if $G=O_n$. Moreover, $\chi(G)=2$ precisely when $G$ is bipartite. 
A result of Constantine~\cite{C85} (see also~\cite{FYW12}) shows that among all $n$-vertex graphs (or all $n$-vertex bipartite graphs), the unique graph minimizing the least eigenvalue is the complete bipartite graph $K_{\lceil n/2\rceil,\lfloor n/2\rfloor}=O_{\lceil n/2\rceil}\vee O_{\lfloor n/2\rfloor}$. 
Accordingly, the present work focuses on graphs with $\chi(G) \ge 3$.

In~\cite[Theorems~3.1 and 3.2]{FYW12}, Fan, Yu and Wang established an upper bound for the chromatic number of a graph in terms of its order and least adjacency eigenvalue, which also appears in Stani\'c's book~\cite[Theorem~3.11]{St15}.

\begin{thm}[\cite{FYW12}]\label{thm:FYW}
Let $G$ be a graph of order $n$ with chromatic number $3\leq \chi \leq \frac{n}{2}$ and least adjacency eigenvalue $\lambda_n$.
Then
$$
\chi \le \left(\frac{n}{2}+1+\lambda_n\right) +
\sqrt{\left(\frac{n}{2}+1+\lambda_n\right)^2 - 4(\lambda_n+1)\left(\lambda_n + \frac{n}{2}\right)},
$$
with equality if and only if $G \cong \left(K_{\frac{\chi}{2}}\cup\tfrac{n-\chi}{2}K_1\right) \vee \left(K_{\frac{\chi}{2}}\cup\tfrac{n-\chi}{2}K_1\right)$, where both $n$ and $\chi$ are even.
\end{thm}

Fan, Yu and Wang~\cite{FYW12} conjectured that the above bound continues to hold for 
$\frac{n}{2}<\chi\le n-1$. In this paper we resolve this conjecture.

Let us briefly explain the new difficulty in the range $\chi>n/2$. After reducing the problem to the two-parameter family $G(a,a_0)$, the proof in~\cite[Section~3]{FYW12} treats $a$ and $a_0$ as real variables and compares $f(a,a_0,\lambda)$ with $f\left(\frac{\chi}{2},\frac{n-\chi}{2},\lambda\right)$ for fixed $\lambda\in(-n/2,-1)$. This comparison is obtained from sign estimates for partial derivatives, and these estimates use inequalities that follow from $\chi\le n/2$. For example, the positivity of $n-\chi+1+\lambda$ follows in that range from $\lambda>-n/2$ and $\chi\le n/2$. When $\chi>n/2$, such sign information is no longer uniform over the feasible region, so the same calculus argument does not seem to yield the desired minimization directly.

Our main observation is that the root comparison can instead be made at a single point. Lemma~\ref{lem:unique} shows that every relevant quartic has exactly one negative real root. Hence, if $\xi$ denotes the negative root of the balanced polynomial $f\left(\frac{\chi}{2},\frac{n-\chi}{2},\lambda\right)$, then it is enough to prove $f(a,a_0,\xi)\ge 0$ for every feasible pair $(a,a_0)$. After centering the parameters at the balanced point, the difference $f(a,a_0,\lambda)-f\left(\frac{\chi}{2},\frac{n-\chi}{2},\lambda\right)$ has an explicit form. Substituting the quadratic equation satisfied by $\xi$ reduces the desired nonnegativity to a short completion-of-squares argument.

\section{Preliminaries}

We follow the notation in~\cite{FYW12}. Given positive integers $n$ and $\chi$ with $3\le \chi\le n-1$, 
we introduce a family of graphs defined by
\[
G(a,a_0) := \left(K_a\cup O_{a_0}\right) \vee \left(K_b\cup O_{b_0}\right),
\]
where 
\begin{equation}\label{eq:feasible1}
1\le a\le \chi-1,\qquad b=\chi-a,\qquad 
a_0,b_0\ge0,\qquad a_0+b_0=n-\chi.
\end{equation}
We call a pair $(a,a_0)$ \emph{feasible} if the integers $a$ and $a_0$ satisfy~\eqref{eq:feasible1}. It is immediate that $\chi(G(a,a_0))=\chi$, and since $G(a,a_0)$ is connected but not complete, we have $\lambda_n(G(a,a_0))<-1$.

By~\cite[Eq.~(2.2)]{FYW12}, 
the least eigenvalue $\lambda_n(G(a,a_0))$ is the smallest real root of the quartic polynomial:
\begin{equation}\label{eq:quartic}
f(a,a_0,\lambda) := \lambda^2(\lambda-a+1)(\lambda-b+1)
-\bigl[(b+b_0)\lambda-b_0(b-1)\bigr]
  \bigl[(a+a_0)\lambda-a_0(a-1)\bigr],
\end{equation}
where $b=\chi-a$ and $b_0=n-\chi-a_0$, and throughout the paper $n$ and $\chi$ are regarded as fixed numbers.

We recall the following structural lemma from~\cite[Lemma~2.1]{FYW12}.

\begin{lem}[\cite{FYW12}]\label{lem:FYW21}
Among all graphs of order~$n$ and chromatic number~$\chi$, where $3\le \chi\le n-1$, if a graph~$G$ is one whose least eigenvalue attains the minimum, 
then $G$ must be of the form $G(a,a_0)$ for some $a$ and~$a_0$.
\end{lem}

We also need the following lemma concerning the number of negative roots of the polynomial $f(a,a_0,\lambda)$ defined in~\eqref{eq:quartic}.

% \begin{lem}\label{lem:unique}
% For any feasible $(a,a_0)$, the polynomial $f(a,a_0,\lambda)$ has exactly one real root in $(-\infty,0)$, and it is simple. Moreover $f(a,a_0,0)\le 0$ and $\lim_{\lambda\to-\infty}f(a,a_0,\lambda)=+\infty$.
% \end{lem}

\begin{lem}\label{lem:unique}
For every feasible pair $(a,a_0)$, the polynomial $f(a,a_0,\lambda)$ has exactly one negative real root (counted with multiplicity). Moreover, $f(a,a_0,0)\le 0$ and $\lim_{\lambda\to-\infty} f(a,a_0,\lambda)=+\infty$.
\end{lem}

\begin{proof}
Write $g(x):=f(a,a_0,-x)$. Let $b=\chi-a$ and $b_0=n-\chi-a_0$. A direct expansion gives
\[
\begin{aligned}
g(x)
&=x^2(x+ a - 1)(x+ b - 1)-((b + b_0)x+ b_0(b - 1))((a + a_0)x+ a_0(a - 1))\\
&=x^4+(a+b-2) x^3+\left((a-1)(b-1)-(a+a_0)(b+b_0)\right)x^2\\
&\qquad-\left((b+b_0)a_0(a-1)+(a+a_0)b_0(b-1)\right)x-a_0(a-1)b_0(b-1).
\end{aligned}
\]
Hence the coefficient signs of $g$ are
\[
+\ ,\quad \ge 0\ ,\quad -\ ,\quad \le 0\ ,\quad \le 0.
\]Ignoring possible zeros, the sign sequence has exactly one sign change (from $+$ to $-$). By Descartes' rule of signs~\cite{Wang2004Descartes}, $g$ has exactly one positive real root with multiplicity $1$. Therefore $f(a,a_0,\lambda)$ has exactly one real root (counted with multiplicity) in $(-\infty,0)$.

Finally, $f(a,a_0,0)=-b_0(b-1)\,a_0(a-1)\le 0$, and since the leading coefficient of $f$ is $+1$, we have $\lim_{\lambda\to-\infty}f(a,a_0,\lambda)=+\infty$.
\end{proof}

\section{Main results}

In this section we extend \cite[Theorems~3.1 and 3.2]{FYW12} to the full range \(3\le \chi\le n-1\).

\begin{thm}\label{thm:main1}
Let $G$ be a simple graph of order $n$ with chromatic number $\chi$ satisfying 
$3\le \chi\le n-1$. Then the least adjacency eigenvalue $\lambda_n$ satisfies
\[
\lambda_n \ge 
-\frac{n-\chi+2+\sqrt{(n-\chi-2)^2+4\chi(n-\chi)}}{4},
\]
with equality if and only if $G \cong \left(K_{\frac{\chi}{2}}\cup\tfrac{n-\chi}{2}K_1\right) \vee \left(K_{\frac{\chi}{2}}\cup\tfrac{n-\chi}{2}K_1\right)$, where both $n$ and $\chi$ are even.
\end{thm}

\begin{proof}
By Lemma~\ref{lem:FYW21}, to determine the extremal value of $\lambda_n(G)$ among all graphs of order $n$ and chromatic number $\chi$, it suffices to minimize the smallest real root of $f(a,a_0,\lambda)$ over all feasible pairs $(a,a_0)$. Our goal is to prove that, for every feasible pair $(a,a_0)$,
\begin{equation}\label{eq:least_text1}
\text{the smallest real root of }f\left(\tfrac{\chi}{2},\tfrac{n-\chi}{2},\lambda\right)
\ \text{is not greater than that of } f(a,a_0,\lambda).
\end{equation}
Let $\xi$ denote the smallest real root of $f\left(\tfrac{\chi}{2},\tfrac{n-\chi}{2},\lambda\right)$.
In view of Lemma~\ref{lem:unique}, each $f(a,a_0,\lambda)$ has exactly one real root
on $(-\infty,0)$; hence Eq.~\eqref{eq:least_text1} is equivalent to showing that
\begin{equation}\label{eq:least_text2}
f(a,a_0,\xi)\ \ge\ 0\qquad\text{for every feasible pair }(a,a_0).
\end{equation}
Indeed, if $f(a,a_0,\xi)\ge 0$, then $\xi$ lies to the left of (or coincides with) the unique negative root of $f(a,a_0,\lambda)$, so the smallest real root of $f(a,a_0,\lambda)$ is at least $\xi$.

Fix $n,\chi$ and write
\begin{equation}\label{eq:compute_LR_1}
p:=\frac{n-\chi}{2}>0,\quad q:=\frac{\chi}{2}-1>0,\quad M:=\frac{n}{2}\lambda,
\end{equation}
\begin{equation}\label{eq:compute_LR_2}
a=\frac{\chi}{2}+t,\quad b=\frac{\chi}{2}-t,\quad
a_0=p+s,\quad b_0=p-s,
\end{equation}
so that $|t|\le q$ and $|s|\le p$. Define
\[
\Delta(s,t;\lambda):=f(a,a_0,\lambda)-f\left(\tfrac{\chi}{2},\tfrac{n-\chi}{2},\lambda\right).
\]
The next claim is a purely algebraic identity.
\begin{claim}\label{clm:Delta-again}
For all $\lambda\in\mathbb{R}$,
\[
\Delta(s,t;\lambda)=p(p-2\lambda) t^2+(\lambda-q)^2 s^2-s^2t^2+2\lambda(\lambda+1) st.
\]
\end{claim}
\begin{proof}[Proof of Claim~\ref{clm:Delta-again}]
Write
\[
L:=(b+b_0)\lambda-b_0(b-1),\qquad R:=(a+a_0)\lambda-a_0(a-1),
\]
a direct calculation with \eqref{eq:compute_LR_1} and \eqref{eq:compute_LR_2} gives
\[
L=(M-pq)-st-\left(s(\lambda-q)+t(\lambda-p)\right),\quad
R=(M-pq)-st+\left(s(\lambda-q)+t(\lambda-p)\right),
\]
hence
\[
LR=(M-pq)^2-2(M-pq)st-\left[(\lambda-q)^2s^2+(\lambda-p)^2t^2+2(\lambda-q)(\lambda-p)st\right]+s^2t^2.
\]Using
\[
(\lambda-a+1)(\lambda-b+1)=(\lambda-q-t)(\lambda-q+t)=(\lambda-q)^2-t^2,
\]we know that\[
f(a,a_0,\lambda)=\lambda^2 \left[(\lambda-q)^2-t^2\right]-LR, \quad f\left(\tfrac{\chi}{2},\tfrac{n-\chi}{2},\lambda\right)=\lambda^2 (\lambda-q)^2 -(M-pq)^2, 
\] it follows that $\Delta(s,t;\lambda) = -\lambda^2 t^2 -LR + (M-pq)^2$. Now\[
\begin{aligned}
\Delta(s,t;\lambda)&= -\lambda^2 t^2 +2(M-pq)st +\left[(\lambda-q)^2s^2+(\lambda-p)^2t^2+2(\lambda-q)(\lambda-p)st\right]-s^2t^2
\\&=\left( (\lambda-p)^2-\lambda^2 \right)t^2 + (\lambda-q)^2s^2+2st \left[(M-pq)+(\lambda-p)(\lambda-q) \right] - s^2t^2.
\end{aligned}
\]Recall that $(\lambda-p)^2-\lambda^2=p(p-2\lambda)$ and \[
(M-pq)+(\lambda-p)(\lambda-q)=\lambda^2-(p+q-\tfrac{n}{2})\lambda=\lambda(\lambda+1),
\]we obtain \[
\Delta(s,t;\lambda)=p(p-2\lambda) t^2+(\lambda-q)^2 s^2-s^2t^2+2\lambda(\lambda+1) st,
\]as desired.
\end{proof}

Recall that $\xi$ is the smallest real root of the polynomial
\[
f\left(\tfrac{\chi}{2},\tfrac{n-\chi}{2},\lambda\right)=\lambda^2(\lambda-q)^2-(M-pq)^2.
\]
By Lemma~\ref{lem:unique}, for each feasible pair $(a,a_0)$, $f(a,a_0,\lambda)$ has exactly one real root on $(-\infty,0)$, and therefore $\xi<0$ is the unique negative root of $f\left(\tfrac{\chi}{2},\tfrac{n-\chi}{2},\lambda\right)$ characterized by
\begin{equation}\label{eq:xi-eqn}
\lambda(\lambda-q)=-(M-pq)\quad\Longleftrightarrow\quad \lambda^2+(1+p)\lambda-pq=0.
\end{equation}
From \eqref{eq:xi-eqn} we will use the convenient identity
\begin{equation}\label{eq:key-identity}
\xi(\xi+1)=p (q-\xi).
\end{equation}
To prove Eq.~\eqref{eq:least_text2}, it suffices to prove
\begin{equation}\label{eq:least_text3}
\Delta(s,t;\xi) \ge 0\qquad\text{for every feasible pair }(a,a_0).
\end{equation}
Define
\[
\mu:=-\xi>0,\qquad \alpha:=p+2\mu=p-2\xi>0,\qquad \beta:=q+\mu=q-\xi>0.
\]By Claim~\ref{clm:Delta-again} and \eqref{eq:key-identity}, we can rewrite
\[
\Delta(s,t;\xi)
=(\beta^2-t^2)s^2+2p\beta t s+p\alpha t^2.
\]
For each fixed $t$, this is a convex quadratic polynomial in $s$ with leading coefficient $\beta^2-t^2>0$ (since $\beta>q\ge |t|$). Completing the square gives
\begin{equation}\label{eq:Delta-min}
\Delta(s,t;\xi)
=(\beta^2-t^2)\left(s+\frac{p\beta t}{\beta^2-t^2}\right)^2
+\frac{pt^2}{\beta^2-t^2}\left[\alpha(\beta^2-t^2)-p\beta^2\right].
\end{equation}
Hence
\begin{equation}\label{eq:min-s_p1}
\min_{|s|\leq p}\Delta(s,t;\xi)\geq\min_{s\in\mathbb{R}}\Delta(s,t;\xi)
=\frac{pt^2}{\beta^2-t^2}\left[\alpha(\beta^2-t^2)-p\beta^2\right].
\end{equation}Thus, to prove Eq.~\eqref{eq:least_text3}, it suffices to prove
\begin{equation}\label{eq:least_text4}
\alpha(\beta^2-t^2)\geq p\beta^2 \qquad\text{for any }|t|\le q.
\end{equation}
In particular, it is enough to verify that $\alpha(\beta^2-q^2)-p\beta^2 \geq 0$. Using $\mu^2=(1+p)\mu+pq$ from \eqref{eq:xi-eqn}, we obtain
\[
\begin{aligned}
\alpha(\beta^2-q^2)-p\beta^2
&=(p+2\mu)\mu(2q+\mu)-p(q+\mu)^2=2\mu^3+4\mu^2 q-pq^2
\\
&=2\mu^3+ 4q((1 + p)\mu + pq)-pq^2\\
&=2\mu^3+4(1+p)\mu q+3pq^2\ >\ 0.
\end{aligned}
\]
Therefore $\min_{|s|\leq p}\Delta(s,t;\xi)\ge 0$ for all $|t|\le q$, this completes the proof of Eq.~\eqref{eq:least_text1}. The least root of the quadratic equation~\eqref{eq:key-identity} is\[
\xi=-\frac{n-\chi+2+\sqrt{(n-\chi-2)^2+4\chi(n-\chi)}}{4}.
\]By Eq.~\eqref{eq:least_text1}, it follows that for every feasible pair $(a,a_0)$, the smallest real root of $f(a,a_0,\lambda)$ is greater than or equal to $\xi$. Hence, by Lemma~\ref{lem:FYW21}, for all graphs of order $n$ with chromatic number $\chi$ satisfying $3\le \chi\le n-1$, the least adjacency eigenvalue $\lambda_n$ satisfies
\begin{equation}\label{eq:least_text100}
\lambda_n \ge 
\xi= -\frac{n-\chi+2+\sqrt{(n-\chi-2)^2+4\chi(n-\chi)}}{4}.
\end{equation}

Now we consider the equality condition. If $\lambda_n(G) = \xi$, then $G$ must be of the form $G(a,a_0)$ for some feasible integers $a$ and~$a_0$ by Lemma~\ref{lem:FYW21}. Hence the smallest real root of $f\!\left(\tfrac{\chi}{2},\tfrac{n-\chi}{2},\lambda\right)$ coincides with that of $f(a,a_0,\lambda)$, which forces $\Delta(s,t;\xi)=0$. Since $\alpha(\beta^2-q^2)-p\beta^2>0$, by \eqref{eq:min-s_p1} the only way for $\Delta(s,t;\xi)=0$ is to have $t=0$. Substituting $t=0$ into $\Delta(s,t;\xi)$ gives
\[
\Delta(s,0;\xi)=(\beta^2-0)s^2=\beta^2 s^2.
\]
Thus $\Delta(s,t;\xi)=0$ implies $s=0$ as well, i.e.\ $(s,t)=(0,0)$.
Equivalently, $a=\tfrac{\chi}{2}$ and $a_0=\tfrac{n-\chi}{2}$. 
This shows that the equality graph in~\eqref{eq:least_text100} is precisely 
\[
G\left(\tfrac{\chi}{2},\tfrac{n-\chi}{2}\right)
=\left(K_{\frac{\chi}{2}}\cup\tfrac{n-\chi}{2}K_1\right) \vee \left(K_{\frac{\chi}{2}}\cup\tfrac{n-\chi}{2}K_1\right),
\]
where both $n$ and $\chi$ are even.
\end{proof}

By the same argument as in the proof of \cite[Theorem~3.2]{FYW12}, 
we derive the following upper bound on the chromatic number in terms of its order and least adjacency eigenvalue.

\begin{thm}\label{thm:chi-upper}
Let $G$ be a graph of order $n$ with chromatic number $3\leq \chi \leq n-1$ and least adjacency eigenvalue $\lambda_n$.
Then
\begin{equation}\label{eq:chi-upper}
\chi \le \left(\frac{n}{2}+1+\lambda_n\right) +
\sqrt{\left(\frac{n}{2}+1+\lambda_n\right)^2 - 4(\lambda_n+1)\left(\lambda_n + \frac{n}{2}\right)},
\end{equation}
with equality if and only if $G \cong \left(K_{\frac{\chi}{2}}\cup\tfrac{n-\chi}{2}K_1\right) \vee \left(K_{\frac{\chi}{2}}\cup\tfrac{n-\chi}{2}K_1\right)$, where both $n$ and $\chi$ are even.
\end{thm}

\section{A comparison with Wilf's bound}
As noted above, Wilf proved that $\chi(G)\le 1+\lambda_1(G)$. In fact, as proved by Wu and Elphick~\cite[Lemma~3.2]{WE17},
\[
\chi(G) \le \chi_\ell(G) \le \operatorname{col}(G) \le 1+\lambda_1(G),
\]
where $\chi_\ell(G)$ denotes the list chromatic number and $\operatorname{col}(G)$ denotes the coloring number. However, in Theorem~\ref{thm:chi-upper} one \emph{cannot} replace $\chi(G)$ by $\chi_\ell(G)$ or by $\operatorname{col}(G)$.

When $\chi = 2$, Theorem~\ref{thm:main1} simplifies to $\lambda_n \ge -n/2$, which was proved by Constantine \cite{C85} for all graphs. Theorem~\ref{thm:chi-upper} is therefore true when $\chi = 2$, and we can let $G=K_{3,3}$ for which $n=6$ and the least adjacency eigenvalue is $\lambda_n(G)=-3$.
Plugging $\lambda_n=-3$ and $n=6$ into the right-hand side of~\eqref{eq:chi-upper} gives
\[
\left(\frac{n}{2}+1+\lambda_n\right)+
\sqrt{\left(\frac{n}{2}+1+\lambda_n\right)^{2}-4(\lambda_n+1)\left(\lambda_n+\frac{n}{2}\right)}
=2.
\]
But since $\operatorname{col}(K_{3,3})\ge \chi_\ell(K_{3,3})=3$, if $\chi$ in Theorem~\ref{thm:chi-upper} were replaced by $\chi_\ell$ or $\operatorname{col}$, it would falsely assert that 
$\chi_\ell(K_{3,3})\le \operatorname{col}(K_{3,3})\le 2$. Therefore the chromatic number in Theorem~\ref{thm:chi-upper} cannot be replaced by the list chromatic number or the coloring number.

Fan, Yu and Wang~\cite{FYW12} demonstrated that their bound and Wilf's bound are incomparable. To exemplify this incomparability, for the connected graphs in the Wolfram Mathematica database of graphs with $n = 16$ and $\chi \ge 3$:

\begin{itemize}
\item Wilf's bound outperforms Fan--Yu--Wang's bound for 333 graphs;
\item Fan--Yu--Wang's bound outperforms Wilf's bound for 23 graphs; and
\item the bounds are equal for 3 graphs.
\end{itemize}
Fan--Yu--Wang's bound performs comparatively well, for example, on some circulant, complete multipartite, and cone graphs.

\section{Extremal graphs}
In addition to proving Theorem~\ref{thm:FYW}, Fan, Yu and Wang also proved the following result~\cite[Theorem~3.3]{FYW12}:
\begin{thm}[\cite{FYW12}]
Among all graphs of order $n$ and with chromatic number $\chi$, where $3 \leq \chi \leq n/2$, the graph $G(\lceil\chi/2\rceil, \lfloor(n - \chi)/2\rfloor)$ is the unique one whose least eigenvalue attains the minimum.
\end{thm}
Fan, Yu and Wang~\cite{FYW12} also conjectured that this theorem remains true for the full range $3\le \chi \le n-1$. This conjecture follows directly from Theorem~\ref{thm:main1} when both $n$ and $\chi$ are even. For the remaining parity cases, our fully algebraic approach appears to be less straightforward than the calculus-based approach used in \cite{FYW12}. At the symmetric point $(a,a_0)=(\tfrac{\chi}{2},\tfrac{n-\chi}{2})$, the characteristic polynomial can be factorized as a difference of squares, and thus the least root $\xi$ satisfies a \emph{quadratic equation}. Substituting this relation into $\Delta$ is straightforward, and the nonnegativity of $\Delta$ can then be easily verified. However, for instance, when both $\chi$ and $n$ are odd, we have $(a,a_0)=(\tfrac{\chi+1}{2},\tfrac{n-\chi}{2})$. In this case, the symmetry is broken, and the least root $\zeta$ instead satisfies a \emph{quartic equation}:
\[
\lambda^2(\lambda-q)^2-\left(\frac{n}{2}\lambda-pq\right)^2+\tfrac14 p (p-2\lambda)=0.
\]
Consequently, proving $\Tilde{\Delta}(s,t;\zeta):=f(a,a_0,\zeta)-f\left(\tfrac{\chi+1}{2},\tfrac{n-\chi}{2},\zeta\right) \ge0$ must be carried out over the half-integer lattice $t\in\tfrac12+\mathbb{Z}$, with separate treatment required for the endpoint cases $t=\tfrac12,\tfrac32,$ and $t=q$ (since we may assume without loss of generality that $a\ge b$, i.e., $t\ge0$). Each of these endpoint cases involves repeated eliminations using the above quartic identity, which is substantially more complicated than the algebraic substitution based on a quadratic equation. We believe the conclusion still holds, although a concise and unified algebraic proof covering all parity cases remains elusive.

\section{Concluding remarks}

In this paper, we resolved a conjecture of Fan, Yu and Wang by proving that their upper bound on the chromatic number also holds in the range $\frac{n}{2}<\chi\le n-1$. We also note that this is a natural high-chromatic regime. For example, it contains all complete multipartite graphs of order $n$ with $k$ nonempty parts, where $\frac{n}{2}<k\le n-1$, since such a graph has chromatic number $k$. More generally, let $H$ be a graph on $h$ vertices, and let $K_r\vee H$ denote the join of $K_r$ and $H$. Then $K_r\vee H$ has order $r+h$ and $\chi(K_r\vee H)=r+\chi(H)$, because the colors used on $K_r$ cannot be reused on $H$. Hence $K_r\vee H$ lies in the range $\chi>n/2$ whenever $r+\chi(H)>\frac{r+h}{2}$. If $H$ is not complete, then $K_r\vee H$ is also non-complete and has chromatic number at most $r+h-1$. Thus the extension from $3\le \chi\le n/2$ to the full range $3\le \chi\le n-1$ covers a broad and varied class of high-chromatic graphs, not only a small near-complete exceptional family.

The parameter $n/2$ in Theorem~\ref{thm:FYW} is closely related to the universal lower bound $\lambda_n\ge -n/2$. A related estimate of Powers~\cite{P89} gives $\lambda_n^2\le m$; indeed, \[ 2m=\sum_{i=1}^n\lambda_i^2\ge \lambda_1^2+\lambda_n^2\ge 2\lambda_n^2. \] Wu and Elphick~\cite{WE17} proved that 
\begin{equation} 
\chi(\chi - 1) \le (\lambda_1 + 1)\lambda_1 \le 2m. 
\end{equation}
By analogy with Theorem~\ref{thm:FYW}, this suggests the following conjecture, in which $m$ replaces $n/2$, $\chi(\chi - 1)$ replaces $\chi$, and $-\lambda_n^2$ replaces $\lambda_n$. Note that this conjectured bound is at most $2m$.

\begin{conj}\label{conj:clive}
For any non-empty graph $G$,\[
\chi(\chi - 1) \le (m + 1 - \lambda_n^2) + \sqrt{(m + 1 - \lambda_n^2)^2 - 4(\lambda_n^2 - 1)(\lambda_n^2 - m)}.
\]
\end{conj}

Conjecture~\ref{conj:clive} is exact for $K_n$ and complete bipartite graphs, and the proof is trivial for bipartite graphs. We have verified Conjecture~\ref{conj:clive} for graphs of order at most 9, and graphs with order at most $100$ in the Wolfram Mathematica database. As for Theorem~\ref{thm:chi-upper}, the chromatic number in Conjecture~\ref{conj:clive} cannot be replaced by the list chromatic number or the coloring number. 

This edge-based conjecture typically performs better than the vertex-based bound proved in this paper. For example for the connected graphs in the Mathematica database with 16 vertices:

\begin{itemize}
    \item Conjecture~\ref{conj:clive} outperforms Fan--Yu--Wang's bound~\eqref{eq:chi-upper} for 384 graphs;
    \item Fan--Yu--Wang's bound~\eqref{eq:chi-upper} outperforms Conjecture~\ref{conj:clive} for 34 graphs; and
    \item the bounds are equal for 2 graphs.
\end{itemize}

\section*{Acknowledgements}

The authors thank Jie Ma for helpful comments and suggestions, and Hitesh Kumar for pointing out several typos in an earlier version of this paper. They are also grateful to the anonymous referees for their careful reading and valuable comments.

\end{document}